\documentclass{amsart}
\usepackage{amsmath}
\usepackage{amssymb}
\usepackage{amsthm}
\begin{document}
\def\eq#1{{\rm(\ref{#1})}}
\theoremstyle{plain}
\newtheorem{thm}{Theorem}[section]
\newtheorem{lem}[thm]{Lemma}
\newtheorem{prop}[thm]{Proposition}
\newtheorem{cor}[thm]{Corollary}
\theoremstyle{definition}
\newtheorem{dfn}[thm]{Definition}
\newtheorem{rem}[thm]{Remark}
\def\Ker{\mathop{\rm Ker}}
\def\Coker{\mathop{\rm Coker}}
\def\ind{\mathop{\rm ind}}
\def\Re{\mathop{\rm Re}}
\def\vol{\mathop{\rm vol}}
\def\SO{\mathbin{\rm SO}}
\def\Im{\mathop{\rm Im}}
\def\min{\mathop{\rm min}}
\def\Spec{\mathop{\rm Spec}\nolimits}
\def\Hol{{\textstyle\mathop{\rm Hol}}}
\def\ge{\geqslant}
\def\le{\leqslant}
\def\Z{{\mathbin{\mathbb Z}}}
\def\R{{\mathbin{\mathbb R}}}
\def\N{{\mathbin{\mathbb N}}}
\def\E{{\mathbb E}}
\def\V{{\mathbb V}}
\def\K{{\mathbb K}}
\def\D{{\mathbb D}}
\def\al{\alpha}
\def\be{\beta}
\def\ga{\gamma}
\def\de{\delta}
\def\ep{\epsilon}
\def\io{\iota}
\def\ka{\kappa}
\def\la{\lambda}
\def\ze{\zeta}
\def\th{\theta}
\def\vp{\varphi}
\def\si{\sigma}
\def\up{\upsilon}
\def\om{\omega}
\def\De{\Delta}
\def\Ga{\Gamma}
\def\Th{\Theta}
\def\La{\Lambda}
\def\Om{\Omega}
\def\ts{\textstyle}
\def\sst{\scriptscriptstyle}
\def\sm{\setminus}
\def\op{\oplus}
\def\ot{\otimes}
\def\bigop{\bigoplus}
\def\iy{\infty}
\def\ra{\rightarrow}
\def\longra{\longrightarrow}
\def\dashra{\dashrightarrow}
\def\t{\times}
\def\w{\wedge}
\def\d{{\rm d}}
\def\bs{\boldsymbol}
\def\ci{\circ}
\def\ti{\tilde}
\def\ov{\overline}
\def\md#1{\vert #1 \vert}
\def\nm#1{\Vert #1 \Vert}
\def\bmd#1{\big\vert #1 \big\vert}
\def\cnm#1#2{\Vert #1 \Vert_{C^{#2}}} 
\def\lnm#1#2{\Vert #1 \Vert_{L^{#2}}} 
\def\bnm#1{\bigl\Vert #1 \bigr\Vert}
\def\bcnm#1#2{\bigl\Vert #1 \bigr\Vert_{C^{#2}}} 
\def\blnm#1#2{\bigl\Vert #1 \bigr\Vert_{L^{#2}}} 
\title[Lagrangian-type submanifolds of $Spin(7)$ manifolds and their deformations]{Lagrangian-type submanifolds of $Spin(7)$ manifolds and their deformations}
\author{Rebecca Glover and Sema Salur}
\thanks{S.Salur is partially supported by NSF grant 1105663}
\address{Department  of Mathematics, University of St. Thomas, St. Paul, MN, 55105}
\email{rebecca.glover@stthomas.edu}
\address {Department of Mathematics, University of Rochester, Rochester, NY, 14627}
\email{sema.salur@rochester.edu } 
\date{\today}

\begin{abstract} In an earlier paper  \cite{GS} we showed that the space of deformations of a smooth, compact, orientable Harvey-Lawson submanifold $HL$ in a $G_2$ manifold $M$ can be identified with the direct sum of the space of smooth functions and closed 2-forms on $HL$. In \cite{GS}, we also introduced a new class of Lagrangian-type 4-dimensional submanifolds inside $G_2$ manifolds, called them RS submanifolds, and proved that the space of deformations of a smooth, compact, orientable $RS$ submanifold in a $G_2$ manifold $M$ can be identified with closed 3-forms on $RS$.  In this short note, we define a new class of Lagrangian-type 4-dimensional submanifolds inside $Spin(7)$ manifolds, which we call $L$-submanifolds.  We show that the space of deformations of a smooth, compact, orientable $L$-submanifold in a $Spin(7)$ manifold $N$ can be identified with the space of closed 3-forms on $L$. 
\end{abstract}
\date{}
\maketitle
\section{Introduction}



The study of Lagrangian submanifolds is an important topic in the field of symplectic geometry.  Lagrangian submanifolds reveal information about Hamiltonian mechanics, symplectic rigidity, and local invariants of symplectic manifolds.  Further, the moduli spaces of Lagrangian submanifolds have been extensively used to attempt to solve Kontsevich's Mirror Symmetry Conjecture.  One may ask if similar objects exist in other geometric settings.  

\vspace{.1in}

In this paper, we consider the exceptional geometry of Spin(7) manifolds and define Lagrangian-type structures in this geometric setting.  Let $(N, \Psi)$ be a $Spin(7)$ manifold.  We define an $L$-submanifold to be a $4$-dimensional submanifold of $N$ such that $\Psi |_{L} = 0$.  Note that this is analogous to the condition for a submanifold of a symplectic manifold to be Lagrangian.  Here, we discuss properties of these objects and prove that the deformation space of $L$-submanifolds is unobstructed and infinite-dimensional.  In particular, we prove the following theorem.

\begin{thm} Suppose $N$ is a $Spin(7)$ manifold such that $T(N)$ admits a non-vanishing $3$-frame field.  The space of infinitesimal deformations of a smooth, compact,
orientable 4-dimensional $L$-submanifold of $N$ within the class of $L$-submanifolds is infinite-dimensional. The deformation space can be identified with closed differential 3-forms on $L$.
\end{thm}

\section{Deformations of $L$-submanifolds}\label{HLsub}

 Before discussing infinitesimal deformations of $L$-submanifolds, we'll first review some basic definitions of $Spin(7)$ manifolds.  For more details, we refer the reader to \cite{HL}.

\vspace{.1in}

\begin{dfn} An 8-dimensional Riemannian manifold $(N,\Psi)$ is called a
{\em Spin(7) manifold} if the holonomy group of its metric
connection lies in $Spin(7)\subset GL(8)$.

\end{dfn} 

\vspace{.1in}

Equivalently, a $ Spin(7) $ manifold is an $8$-dimensional
Riemannian manifold with a triple cross product $\times $ on its
tangent bundle, and a closed $4$-form $\Psi \in \Omega^{4}(N)$
with
$$ \Psi (u,v,w,z)=\langle u \times v \times w,z \rangle.$$

\begin{dfn}  A 4-dimensional submanifold $X$ of a $Spin(7)$ manifold $(N,\Psi)$ is
called {\em Cayley} if $\Psi|_X\equiv vol(X)$. 
\end{dfn}

\vspace{.05in}

Analogous to the $G_2$ case, we introduce a tangent bundle-valued
$3$-form, which is simply the triple cross product of $N$.

\begin{dfn}  Let $(N, \Psi )$ be a $Spin(7)$ manifold. Then $\Upsilon \in
\Omega^{3}(N, TN)$ is the tangent bundle-valued 3-form defined by
the identity:
\begin{equation*}
\langle \tau (u,v,w) , z \rangle=\Psi  (u,v,w,z)=\langle
u\times v\times w , z \rangle .
\end{equation*}
\end{dfn} 

\vspace{.05in}

\begin{dfn}
A submanifold $L \subset N$ is called an {\it L-submanifold} if $\Psi|_L = 0$.
\end{dfn}

\vspace{.05in}

\begin{rem}
 Let $(N,\omega, \Omega)$ be a complex
4-dimensional Calabi-Yau manifold with K\"{a}hler form $\omega$
and a nonvanishing holomorphic (4,0)-form $\Omega$.  In this particular case, $N$ has a $Spin(7)$ structure given as $ \Psi= \Re \Omega + \frac{1}{2} \omega \wedge \omega$. Using these structures, it is easy to show that all special Lagrangian submanifolds of $N$ (with phase $\theta=\frac{\pi}{2}$) will be also $L$-submanifolds of $N$.

\end{rem}

\begin{rem}
$Spin(7)$ manifolds can also be constructed from $G_2$
manifolds. Let $(M,\varphi)$ be a $G_2$ manifold with a 3-form
$\varphi$, then $M\times S^1$ (or $M\times \mathbb{R}$) has holonomy group $G_2\subset Spin(7)$, hence is
a Spin(7) manifold. In this case $\Psi= \varphi\wedge dt +
*_{7} \varphi$, where $*_{7}$ is the Hodge star operator of $M^7$. This implies that all objects of the form HL$\times \mathbb{R}$, where HL are Harvey-Lawson submanifolds of $M$ (see \cite{AS1}, \cite{GS}), are $L$-submanifolds of $M\times \mathbb{R}$. 
\end{rem}

We now form an analogous construction for $Spin(7)$ manifolds as we did for $G_2$ manifolds \cite{AS1}.  Let $(N, \Psi)$ be a $Spin(7)$ manifold.  In this paper, we'll only consider $Spin(7)$ manifolds $N$ such that $T(N)$ admits a nonvanishing $3$-frame field $\Lambda =\langle u,v,w \rangle$.  Under this assumption, we can decompose $T(N)={\K}\oplus{\D}$, where ${\K}=\langle u,v,w, u\times v\times w \rangle$ is the bundle of Cayley $4$-planes  (where $\Psi$ restricts to be 1) and ${\bf D}$ is the complementary subbundle (note that this is also a bundle of Cayley $4$-planes since the form $\Psi$ is self dual). 
On a chart in $N$ let  $e_1,...e_8$ be an orthonormal frame and $e^1,...,e^8$ be the dual coframe.  Then the calibration
4-form is given as (c.f. \cite{HL}) 
\begin{equation}
\begin{aligned}
\Psi=\;\; e^{1234}&+(e^{12}-e^{34})\wedge (e^{56}-e^{78})\\
&+(e^{13}+e^{24})\wedge (e^{57}+e^{68})\\
&+(e^{14}-e^{23})\wedge (e^{58}-e^{67})+e^{5678},\\
\end{aligned}
\end{equation}

\noindent which is a self dual $4$-form, and the corresponding tangent bundle-valued 3-form is

\begin{equation*}
\begin{aligned}
\tau=&\;\;(e^{234}+e^{256}-e^{278}+e^{357}+e^{368}+e^{458}-e^{467})e_1\\
&+(-e^{134}-e^{156}+e^{178}+e^{457}+e^{468}-e^{358}+e^{367})e_2\\
&+(e^{124}-e^{456}+e^{478}-e^{157}-e^{168}+e^{258}-e^{267})e_3\\
&+(-e^{123}+e^{356}-e^{378}-e^{257}-e^{268}-e^{158}+e^{167})e_4\\
&+(e^{126}-e^{346}+e^{137}+e^{247}+e^{148}-e^{238}+e^{678})e_5\\
&+(-e^{125}+e^{345}+e^{138}+e^{248}-e^{147}+e^{237}-e^{578})e_6\\
&+(-e^{128}+e^{348}-e^{135}-e^{245}+e^{146}-e^{236}+e^{568})e_7\\
&+(e^{127}-e^{347}-e^{136}-e^{246}-e^{145}+e^{235}-e^{567})e_8.\\
\end{aligned}
\end{equation*}

\vspace{.1in}

Since $\Psi$ is self-dual, L-submanifolds are not in the complement space of Cayley submanifolds (unlike their $G_2$ counterparts Harvey-Lawson submanifolds).  However, similar to Harvey-Lawson manifolds, we can use a vector-valued form to define them.

There is a natural vector bundle $V$ on a $Spin(7)$ manifold given by $V = E \times_{\rho} \mathbb{R}^7$, where $E$ is the principal coframe bundle on $N$ and $\rho$ is the representation of $Spin(7)$ on $\mathbb{R}^7$ induced from the standard representation of $SO(7)$ on $\mathbb{R}^7$.  Further, we can define a $V$-valued $4$-form $\eta$ that vanishes identically on the Cayley submanifolds of $N$.  As shown in \cite{mclean}, this form, $\eta$, can be locally given by the components

\begin{eqnarray*}( -e^{1567}-e^{1347} - e^{1246} + e^{1235} + e^{2348} + e^{2568} + e^{3578} - e^{4678}), \\
(-e^{2567} - e^{2347} + e^{1236} + e^{1245}- e^{1348}- e^{1568} + e^{4578}+ e^{3678}), \\
(-e^{3567} + e^{1237} + e^{2346} + e^{1345}+ e^{1248} - e^{4568} - e^{1578}- e^{2678}), \\
(-e^{4567} +e^{1247} - e^{1346} + e^{2345}- e^{1238}+ e^{3568} - e^{2578} + e^{1678}), \\
(-e^{3457} + e^{1257} - e^{2456} - e^{1356}+ e^{1268} - e^{3468}+ e^{1378} + e^{2478}),  \\
(-e^{3467}+ e^{1267}- e^{2356}+ e^{1456}- e^{1258} + e^{3458} - e^{1478}+ e^{2378}), \\
(e^{2467} + e^{1367} - e^{2357} + e^{1457}- e^{1358} - e^{2458} + e^{1468}- e^{2368}). \end{eqnarray*}
One can check that a four-dimensional manifold is an $L$-submanifold if and only if 
\[ \langle \eta|_{L},  \eta|_{L} \rangle = 1.\]
Note that this is equivalent to the Cayley inequality as given in \cite{HL}:
\[ \Psi(u,v,w,z)^2 + |\eta( u\times v\times w\times z)|^2 = |u\wedge v\wedge w\wedge z |^2 .\]

\vspace{.1in}

We now describe the normal bundle of an $L$-submanifold inside the $Spin(7)$ manifold $N$. Recall that an orthonormal $4$-frame field $<u,v,w, u\times v\times w>$ on $(N,\Psi )$ is called a {\it Cayley-frame field}.  Again, we assume there exists a nonvanishing $3$-frame field $\Lambda = <u,v,w>$ on $N$.


 \vspace{.1in}
 
 

Given such a $3$-plane field $u,v,w$, we define $R_{uvw}$ to be the vector field given by the triple cross product
$$R_{uvw}=\tau(u,v,w)=u\times v\times w .$$ 

\pagebreak

\begin{lem} The following properties hold:

\begin{itemize}
\item[(a)]  If  $\E=< u,v,w,z>$ is an $L$-plane field, then $$\V=< R_{uvw}, R_{uvz}, R_{uwz}, R_{vwz}>$$ is also an $L$-plane field. 
\item[(b)] $\E \perp \V$.
 \item[(c)] $\{u,v,w, z, R_{uvw}, R_{uvz}, R_{uwz}, R_{vwz} \}$ is an orthonormal frame field on $N$.
\end{itemize}
\end{lem}

 \vspace{.1in}

In particular we can express $\Psi$ as

\begin{eqnarray*} \Psi & = & u\wedge v\wedge w \wedge (u\times v\times w) \\
& & + [ u\wedge v- w\wedge (u\times v \times w)] \wedge [ z\wedge (u\times v\times z) - (u\times w\times z) \wedge (v\times w \times z)] \\
& & + [u\wedge w + v\wedge (u\times v \times w)] \wedge[ z\wedge (u\times w \times z) + (u\times v\times z)\wedge (v\times w\times z) ]\\
& & + [u\wedge (u\times v\times w) - v\wedge w ] \wedge [ z\wedge (v\wedge w \wedge z) - (u\wedge v\wedge z) \wedge (u\wedge w\wedge z)] \\
& & + z \wedge (u\wedge v\wedge z) \wedge (u\wedge w\wedge z) \wedge (v\wedge w\wedge z) .\end{eqnarray*}
Rewriting this form in terms of our vector fields, we have

\begin{eqnarray*} \Psi &= &u\wedge v\wedge w \wedge R_{uvw}\\
& & + [u\wedge v-w\wedge R_{uvw}] \wedge [ z\wedge R_{uvz} - R_{uwz} \wedge R_{vwz} ] \\
& & + [u\wedge w + v\wedge R_{uvw}] \wedge[ z\wedge R_{uwz} + R_{uvz} \wedge R_{vwz} ] \\
& & + [u\wedge R_{uvw} - v\wedge w ] \wedge [z\wedge R_{vwz} - R_{uvz} \wedge R_{uwz}]\\
& & + z\wedge R_{uvz} \wedge R_{uwz} \wedge R_{vwz} .\end{eqnarray*}
Since 
\[ \langle \tau (u,v,w) , z \rangle=\Psi  (u,v,w,z)=\langle
u\times v\times w , z \rangle=0\] 
for an $L$-submanifold, there exists an isomorphism between the tangent bundle $T(L)$ of an $L$-submanifold and its normal bundle $N(L)$, where $N(L)$ is generated by vector fields $R_{uvw}, R_{uvz}, R_{uwz}$, and $R_{vwz}$.  The vector-valued $3$-form $\tau$ induces this isomorphism.

\vspace{.1in}

Following \cite{GS}, we now discuss deformations of $L$-submanifolds in a $Spin(7)$ manifold.

\vspace{.1in}

\begin{thm} The space of infinitesimal deformations of a smooth, compact,
orientable 4-dimensional $L$-submanifold of a $Spin(7)$
manifold $N$ within the class of $L$-submanifolds is infinite-dimensional.  Further, the deformation space can be identified with the space of closed 3-forms on $L$.
\label{thm1}
\end{thm}

\begin{proof} The deformation map, $F$, for a small vector field $V$ is defined from the space of sections of the normal bundle,
$\Gamma(N(L))$, to the space of differential $3$-forms,
$\Lambda^3T^*(L)$, such that
\begin{equation*}
\begin{split}
&F: \Gamma(N(L))\rightarrow
\Lambda^3T^*(L) , \\
&F(V)=((\exp_V)^*(\Psi|_{L_V})).
\end{split}
\end{equation*}
Further, the map $F$ is the restriction of $\Psi$ to $L_{V}$ and then pulled back to $L$ via $(\exp_V)^*$, where
$\exp_V$ is the normal exponential map that takes $L$ diffeomorphically
onto its image $L_V$ in a neighborhood around 0.

\vspace{.1in}
We can naturally identify normal vector fields to $L$ with differential
$3$-forms on $L$.  Further, since $L$ is compact, these normal
vector fields can be identified with nearby submanifolds.  Thus, the kernel of $F$ corresponds to deformations of $L$.

\vspace{.1in}

The linearization of $F$ at $(0)$ is given by
\begin{equation*}
dF(0):\Gamma(N(L))\rightarrow \Lambda^3T^*(L)
\end{equation*}
\noindent where
\begin{equation*}
\begin{split}
dF(0)(V)=&\frac{\displaystyle\partial}{\displaystyle\partial{t}}F(tV)|_{t=0}=\frac{\displaystyle\partial}{\displaystyle\partial{t}}[\exp_{tV}^*(\Psi)] \\
=&[{\mathcal L}_{V}(\Psi)|_{L}].
\end{split}
\end{equation*}

\vspace{.1in}

Further, by Cartan's formula, we have
\begin{equation*}
\begin{split}
dF(0)(V)&=((i_Vd\Psi +d(i_V\Psi))|_{L}\\
&=d(i_V\Psi)|_{L}=(d\star v),
\end{split}
\end{equation*}
where $i_V$ represents the interior derivative, $v$ is the dual
$3$-form to the vector field $V$ with respect to the induced
metric, and $\star v$ is the Hodge dual of $v$ on L. Hence
\begin{equation*}
dF(0)(V) =(d\star v)=(d^*v).
\end{equation*}
We can now identify the space of nontrivial deformations of $L$-submanifolds with closed $3$-forms on $L$. 

\end{proof}

\begin{rem}  Recall that for a Lagrangian submanifold $S$ of $(N^{2n}, \omega)$, there exists an almost complex structure $J$ on $N$ such that $J$ maps vectors on $S$ to vectors orthogonal to $S$.  In other words, 
\[ \langle Ju, v \rangle = 0 \]
for all $u, v\in TS$.  As a final remark, we note that we can construct analogous structures for Harvey-Lawson and RS submanifolds (as submanifolds of a $G_2$ manifold) and $L$-submanifolds (as submanifolds of a $Spin(7)$ manifold).  

Given a Harvey-Lawson submanifold $HL$ of a $G_2$ manifold $(M, \varphi)$, we can show that the cross product $\times$ acts as this almost complex structure on $M$.  Further, given an RS-submanifold of $M$, the vector-valued $3$-form $\chi$, given by $\chi(u,v,w) = -u\times(v\times w)$ assumes this role, taking tangent vectors on the RS submanifold to normal vectors. Finally, given an $L$-submanifolds of a $Spin(7)$ manifold $(N, \Psi)$, the vector-valued 3-form $\tau$ takes tangent vectors on the $L$-submanifold to normal vectors and thus acts as an almost complex structure in this case.
\end{rem}

{\small{\it Acknowledgements.} Some of this work was completed while the authors were visiting BIRS.  We would like to thank BIRS and the organizers of the Women in Geometry workshop.  This is the second part of work that was done when the second author was visiting Cornell University as the Ruth I. Michler Fellow during the Spring of 2015. Many thanks to the mathematics department at Cornell for their hospitality and AWM for their support during the course of this work.}


\begin{thebibliography}{[FP]}


\bibitem{AS1} Akbulut, S. and Salur, S., {\em Calibrations and manifolds with special holonomy}, Real and Complex Submanifolds, Daejeon, Korea, August 2014, 
Springer Proceedings in Math and Stat. (2015) 505-514.


\bibitem{AS2} Akbulut, S. and Salur, S., {\em Mirror Duality via $G_2$ and $Spin(7)$
Manifolds} (with S. Akbulut)  Special Volume For the Arithmetic $\&$
Geometry Around Quantization, Prog. in Mathematics Series, Birkhauser, (2009), Edited by Yuri Manin $\&$ Matilde Marcolli. 

\bibitem{AS3} Akbulut, S. and Salur, S., {\em Calibrated
Manifolds and Gauge Theory}, J. Reine Angew. Math. {\bf 625}, (2008), 187--214.

\bibitem{AS4} Akbulut, S. and Salur, S., {\em Associative Submanifolds of $G_2$ Manifolds}, Advances in Math. {\bf 217}, (2008), {\bf no. 5}, pp. 2130--2140.


\bibitem{bryant3} Bryant, R.L. {\em Metrics with exceptional holonomy}, Ann. of Math. (2) {\bf 126} (1987), pp. 525-576.

\bibitem{bryant2} Bryant, R.L. {\em  Calibrated embeddings in the special Lagrangian and coassociative cases}, Special issue in memory of Alfred Gray (1939--1998). Ann. Global Anal. Geom. {\bf 18}, {\bf no. 3-4}, (2000) 405--435.

\bibitem{bryant4} Bryant, R.L.  {\em Some remarks on $G_2$-structures}, Proceedings of Goškova Geometry-Topology Conference (2005), 75-109.
 



\bibitem{GS} Glover, R.  and Salur, S., {\em Deformations of Lagrangian Type Submanifolds inside $G_2$ manifolds},  preprint arXiv:1503.03056.

\bibitem{HL} Harvey, F.R. and Lawson, H.B. {\em Calibrated Geometries}, Acta. Math. {\bf 148} (1982),
47-157.



\bibitem{mclean} McLean, R.C. {\em Deformations of calibrated submanifolds}, Comm. Anal. Geom. {\bf 6} (1998),
705-747.







\bibitem{ET} Thomas, E. {\em Postnikov Invariants and Higher Order Cohomology Operations}, Ann. of Math. (2) {\bf 85} (1967), 184-217.
\end{thebibliography}
\end{document}